\documentclass[12pt, reqno]{amsart}
\usepackage{ amsthm,latexsym, amscd, amsfonts, amssymb, graphicx, color}
\usepackage[bookmarksnumbered, colorlinks, plainpages, backref]{hyperref}

\definecolor{TITLE}{rgb}{0.0,0.0,1.0}
\definecolor{AUTHOR1}{rgb}{0.00,0.59,0.00}
\definecolor{AUTHOR2}{rgb}{0.50,0.00,1.00}
\definecolor{SECTION}{rgb}{0.50,0.00,1.00}
\definecolor{FOOTTITLE}{rgb}{0.00,0.50,0.75}
\definecolor{THM}{rgb}{0.7,0.3,0.3}
\definecolor{SEC}{rgb}{0.6,0.1,.5}

\makeatletter
\@namedef{subjclassname@2010}{%
  \textup{2010} Mathematics Subject Classification}
\makeatother
\newtheorem{theorem}{{\color{THM} Theorem}}[section]

\let\oldthm\thetheorem
\def\thetheorem{{\color{THM}\oldthm}}

\newtheorem{lemma}[theorem]{{\color{THM}Lemma}}
\newtheorem{proposition}[theorem]{{\color{THM}Proposition}}
\newtheorem{corollary}[theorem]{{\color{THM}Corollary}}
\theoremstyle{definition}
\newtheorem{definition}[theorem]{{\color{THM}Definition\ }}

\numberwithin{equation}{section}

\newcommand{\A}{\mathfrak A}

\newcommand{\bea}{\begin{eqnarray*}}
\newcommand{\eea}{\end{eqnarray*}}
\numberwithin{equation}{section}
\begin{document}
\title[A Type Of GNS-Construction...]{{\color{TITLE} {\large\bf A TYPE OF GNS-CONSTRUCTION  FOR Banach ALGEBRAS} }}
\author[ Khodsiani , Rejali ]{\color{AUTHOR2} Bahram Khodsiani$^1$ and  Ali  Rejali$^2$ }
\address{$^{1,2}$ Department of Mathematics, University of Isfahan, Isfahan, IRAN.
}
\email{$^{1}$\textcolor[rgb]{0.00,0.00,0.84}{${\texttt {b}}_{-}$khodsiani@sci.ui.ac.ir}}
\email{$^{2}$\textcolor[rgb]{0.00,0.00,0.84}{rejali@sci.ui.ac.ir}}



\begin{abstract} We show that every Banach algebra $\mathfrak{A}$ admits a
representation on a  certain Banach space $E$. In particular,  any Banach algebra $\A$ contained in autoperiodic functionals on $\A$ such that separate the points of $\A$ could be imbedded in $B(E)$ for some reflexive  Banach space  $E$.
 \end{abstract}
\maketitle

\section{\color{SEC}Introduction and Preliminaries}
Let $\mathfrak{A}$ be a Banach algebra with identity.  Consider the mapping $\pi :\mathfrak{A}\rightarrow B(\mathfrak{A})$ by $\pi(a)=L_a$, where $L_a(x)=ax$. Then $\pi$ is an isometric representation on $\mathfrak{A}$, this means that every Banach algebra has  an isometric representation on itself, however $\mathfrak{A}$   may be  not reflexive.

 By the $GNS$-construction, every $C^*$-algebra admits an isometric
representation on  a Hilbert space, see ,\cite{GNS,Seg}.
N.J. Young\cite{y1} and S. Kaijser\cite{kas} showed that a Banach algebra has a faithful representation on a reflexive Banach space, if the weakly almost periodic functionals over $\mathfrak{A}$, denoted by $\mathcal{\mathcal{WAP}}(\mathfrak{A})$, separates the points of $\mathfrak{A}$.  Similarly, if the unit ball of $\mathcal{WAP}(\mathfrak{A})$ is a normming set for $\mathfrak{A}$, then it has an isometric representation on a reflexive Banach space. A Banach algebra $\mathfrak{A}$ is called weakly almost periodic algebra($WAP$-algebra) if it has a bounded below (by renormming isometric) representation on a reflexive Banach space.  This kind of algebras are a good generalization of Arens regular Banach algebras. The class  of $WAP$-algebra consist of dual Banach algebras, Arens regular (in particular, $C^*$-algebras), semi-simple algebras, group algebras.
All of this algebras are studied in Harmonic analysis.

In the following, we show that every Banach algebra $\mathfrak{A}$ admits a
representation on a  certain Banach space $E$. In
\cite{y1,palmer}, they specify $E$, as a reflexive Banach space,
whenever $\mathfrak{A}$ is $WAP$-algebra. Also refer to
\cite[Theorem4.10]{kas} for  the case $\mathfrak{A}$ is Arens
regular.

In \cite{Bkh}  we studied those conditions under which  $M_b(S,\omega)$ is a WAP-algebra (respectively  dual Banach algebra). In particular, $M_b(S)$ is a WAP-algebra (respectively  dual Banach algebra) if and  only if $wap(S)$ separates the points of $S$ (respectively $S$  is {\it compactly cancellative semigroup}).
     In \cite{Bkh1} for a locally compact foundation semigroup $S$, we showed that the absolutely continuous  semigroup measure algebra  $M_a(S)$   is a {\rm WAP}-algebra if and only if the measure algebra $M_b(S)$ is so.  We  also showed that  the second dual of a Banach algebra $\A$ is a {\rm WAP}-algebra, under each Arens products, if and only if $\A^{**}$ is a dual Banach algebra. This is  equivalent to the Arens regularity of $\A$.



\section{\color{SEC} Autocompact Sets and Autoperiodic Functionals}

Let $B$ be an absolutely convex subset of a real or complex vector space $E$.
We denote by $E_B$ the quotient space $(\cup_{\lambda>0}\lambda B)/(\cap_{\lambda>0}\lambda B)$. The gauge (or Minkowski functional) of $B$, which is a seminorm on $\cup_{\lambda>0}\lambda B$, induces a norm $|| .||_B$ on $E_B$.
We shall say that B is autocompact if the unit ball of $\widehat{E_B}$ is weakly compact,
where $\widehat {E_B}$ denotes the completion of $E_B$. The polar of $B$ is denoted by $B^\circ$ and defined by $$\{f\in E^*:|\langle f,x\rangle|\leq1\quad(x\in B)\}$$
The next Lemma is routine.
\begin{lemma}
Let $\A$  be a Banach algebra and let $a\in \A$,  $\lambda>0$ be a real number and  $C\subseteq \A^*$ be a $w^*$-closed and absolutly convex set with polar $C^{\circ}$. Then

\begin{enumerate}
\item $(\lambda C)^\circ=\frac{1}{\lambda}C^\circ$,
\item $a(Ca)^\circ\subseteq C^\circ$,
\item $Ca\subseteq \lambda C$ if and only if   $aC^\circ\subseteq\lambda C^\circ$.
\end{enumerate}
\end{lemma}

\begin{lemma}\label{rr} Let $(E, E^*)$ be a dual pair of vector spaces and let $B$ be a $\sigma(E, E^*)$ closed absolutely convex subset of $E$. Then the following are equivalent:
\begin{enumerate}
  \item $B$ is autocompact;
  \item the completion of $E_B$ is reflexive;
  \item for all pairs of sequences $(x_n)$ in $B$ and $(y^*_m)$ in the polar $B^\circ$ of $B$ in $E^*$ the two repeated limits of the double sequence $(\langle x_n, y^*_m\rangle )$ are equal
whenever they both exist.
\end{enumerate}
\end{lemma}
\begin{proof}Refer to \cite[Lemma5]{y1}.
\end{proof}
Suppose that $\mathfrak{A}$ is a normed algebra and that the $w^*$-closed absolutely convex set $C \subseteq \mathfrak{A}^*$ has polar $C^\circ$ in $\mathfrak{A}$. Then the operation of multiplication on the left by $a \in \mathfrak{A}$ induces a linear mapping $a_C$ on $\hat{\mathfrak{A}_{C^\circ}}$ if and only if $\cap_{\lambda>0}\lambda C^\circ$ is invariant under $a$, or equivalently $Ca \subseteq \lambda C$ for some $\lambda>0$. When this is so we have (denoting the gauge of $C^\circ$ by $p_{C^\circ}$)

\begin{eqnarray*}
 ||a_C||&=& \sup\{p_{C^\circ}(ax):x\in C^\circ\} \\
   &=&\sup\{\inf\{\lambda: |\langle ax, y\rangle |\leq\lambda, \quad x\in  C^\circ\}:\quad \quad y\in C\quad\} \\
   &=&\sup\{\sup\{\lambda: |\langle x, ya\rangle |>\lambda, \quad x\in  C^\circ\}:\quad\quad\mbox{ for some  } \quad y\in C\}\\
    &=&\sup\{\lambda: |\langle x, ya\rangle |>\lambda, \quad\mbox{ for some  }\quad x\in  C^\circ\quad\mbox{and  } \quad y\in C\quad\} \\
   &=& \inf\{\lambda>0: |\langle x, ya\rangle| \leq\lambda\quad,\quad\mbox{ for some  } \quad x \in  C^\circ\quad\mbox{ and}\quad y\in C\quad\} \\
   &=&\inf\{\lambda>0: Ca\subseteq \lambda C\}.
\end{eqnarray*}

Consider now a normed algebra $\mathfrak{A}$. We shall say that $h \in \mathfrak{A}^*$  is right-autoperiodic on $\mathfrak{A}$ if the $\sigma(\mathfrak{A}^*,\mathfrak{A})$-closure of $o_r(h):=\{h.a: || a ||_\mathfrak{A}\leq  1\}$ is autocompact in $\mathfrak{A}^*$. Likewise $h$ is left-autoperiodic if the $\sigma(\mathfrak{A}^*,\mathfrak{A})$-closure of $o_l(h):=\{a.h: || a ||_\mathfrak{A}\leq  1\}$
is autocompact. We call h autoperiodic if it is either right- or left-autoperiodic. The set of all autoperiodic functionals are denoted  by $O(\mathfrak{A})$. Every autoperiodic functional h is weakly almost periodic, for
$$\{h.a: || a ||_{\mathfrak{A}}\leq1\}$$
is strongly bounded in $\mathfrak{A}^*$ and its $\sigma(\mathfrak{A}^*,\mathfrak{A})$-closure is $\sigma(\mathfrak{A}^*,\mathfrak{A})$-compact and
therefore completing, so that it is weakly compact. For $h\in O(\A)$ and $\lambda\in \mathbb{C}$, $\lambda h\in O(\A)$. It is not clear whether the sum of two right autoperiodic functionals is right autoperiodic.

\begin{definition}
 Let $\mathfrak{A}$ be a Banach algebra and $0\not= h\in \mathfrak{ A}^*$.Then  $\mathfrak{A}^*$ is a Banach  $\mathfrak{A}$- module in
canonical fashion.
\begin{enumerate}
    \item [(i)]$\|x\|_{l,h}:=\|x.h\|$, for all $x\in A$.
    \item [(ii)]$\|x\|_{r,h}:=\|h.x\|$, for all $x\in A$.
    \item [(iii)] $B\subseteq \mathfrak{A}$ is $h$-left bounded, if $(B,\|.\|_{l,h})$ is bounded.
    \item [(iv)]$B\subseteq \mathfrak{A}$ is $h$-right bounded, if $(B,\|.\|_{r,h})$ is bounded.
    \item[(v)]$B\subseteq \mathfrak{A}$ is $h$- bounded if $B$ is $h$-left and right bounded .
\end{enumerate}
\end{definition}

\begin{lemma}

Let $\mathfrak{A}$ be a Banach algebra and $0\not= h\in \mathfrak{A}^*$.
Then
\begin{enumerate}
    \item [(i)]The maps $x\mapsto \|x\|_{l,h}$  and $x\mapsto \|x\|_{r,h}$ are semi-norms on $\mathfrak{A}$.
    \item [(ii)]If $N_{l,h}:=\{x\in \mathfrak{A}:\|x\|_{l,h}=0\}$, then $N_{l,h}$ is a closed left ideal in $\mathfrak{A}$.
    \item [(iii)] If $N_{r,h}:=\{x\in \mathfrak{A}:\|x\|_{r,h}=0\}$, then $N_{r,h}$ is a closed right ideal in $\mathfrak{A}$.
    \item [(iv)]$\|x+N_{l,h}\|:=\|x\|_{l,h}$ defines a norm   on $\mathfrak{A}/{N_{l,h}}$.
    \item[(v)]$\|x+N_{r,h}\|:=\|x\|_{r,h}$ defines a norm   on $\mathfrak{A}/{N_{r,h}}$.
\end{enumerate}
\end{lemma}
\begin{proof}
�(i)�For all �$x,y\in \A$ and � �$\lambda\in \mathbb{C}$,�    �$\|x+y\|_{l,h}\leq\|x\|_{l,h}+\|y\|_{l,h}$ and � �$\|\lambda x\|_{l,h}=|\lambda|.\|x\|_{l,h}$. �So $\|.\|_{l,h}$�  is a seminorm on �$\A$. Similarly $\|.\|_{r,h}$�is a seminorm.

�(ii) $N_{l,h}:=\{x\in \A:\|x\|_{l,h}=0\}$�is a linear space. Since for �$x,y\in N_{l,h}$ and $\lambda\in\mathbb{C}$� �$\|x+\lambda y\|_{l,h}\leq\|x\|_{l,h}+|\lambda|\|y\|_{l,h}=0$. Thus � �$x+\lambda y\in N_{l,h}$.� If �$a\in \A$� and   �$x\in N_{l,h}$ then �
�\begin{eqnarray*}�
�\|ax\|_{l,h}=\|ax.h\|&=&\sup\{\langle y,ax.h\rangle:y\in\A�, �||y||\leq1\}\\�
�&=&\sup\{\langle y.a,x.h\rangle:y\in\A�, �||y.a||\leq||a||\}\\�
�&\leq&||a||.||x.h||=0�
�\end{eqnarray*}�
so �$ax\in N_{l,h}$.�

�(iii) is similar to� (ii).�

�(iv) $\|x+N_{l,h}\|:=\|x\|_{l,h}=0$� implies  �$x+N_{l,h}=N_{l,h}$.��So $\|x+N_{l,h}\|:=\|x\|_{l,h}$�defines a norm on �$\A/{N_{l,h}}$.�

�(v)� is similar to �(iv).

\end{proof}
\begin{definition}
Let $\mathfrak{A}$ be a Banach algebra and $0\not= h\in \mathfrak{A}^*$.
Then $\mathfrak{A}_{l,h}:=\overline{(\mathfrak{A}/{N_{l,h}})}$is called
left complemented of $\mathfrak{A}/{N_{l,h}}$ with respect to $h\in
\mathfrak{A}^*$, by  semi-norm $\|.\|_{l,h}$. Also $\mathfrak{A}_{r,h}:=\overline{(\mathfrak{A}/{N_{r,h}})}$is called right
complemented of $\mathfrak{A}/{N_{r,h}}$ with respect to $h\in \mathfrak{A}^*$, by semi-norm $\|.\|_{r,h}$.
\end{definition}
\begin{lemma}\label{zxc}

Let $\mathfrak{A}$ be a Banach algebra and $0\not= h\in \mathfrak{A}^*$.
Then
\begin{enumerate}
    \item [(i)]The map  $\pi_{l,h}:\mathfrak{A}\rightarrow B(\mathfrak{A}_{l,h}),\quad a\mapsto a_{l,h},\quad
    a_{l,h}(x+N_{l,h})=ax+N_{l,h}$ is a continuous representation
    of $\mathfrak{A}$.
    \item [(ii)] The map $\pi_{r,h}:\mathfrak{A}\rightarrow B(\mathfrak{A}_{r,h}),\quad a\mapsto a_{r,h},\quad
    a_{r,h}(x+N_{r,h})=ax+N_{r,h}$ is a continuous anti-representation
    of $\mathfrak{A}$.
    \item [(iii)] If $\mathfrak{A}_l:=\oplus_{h\in \mathfrak{A}^*} \mathfrak{A}_{l,h}$, then $\mathfrak{A}_h$ with  point wise multiplication
    and the norm
     $$\|(a_h)\|_l=(\sum_{h\in \mathfrak{A}^*}\|a_h\|^2_{l,h})^{\frac{1}{2}}$$ is a Banach space.
    \item [(iv)]If $\mathfrak{A}_r:=\oplus_{h\in \mathfrak{A}^*} \mathfrak{A}_{r,h}$, then $\mathfrak{A}_r$ with point wise multiplication
    and the norm
     $$\|(a_h)\|_r=(\sum_{h\in \mathfrak{A}^*}\|a_h\|^2_{r,h})^{\frac{1}{2}}$$ is a Banach space.
    \item[(v)]The map $\pi_l:\mathfrak{A}\rightarrow B(\mathfrak{A}_l),\quad a\mapsto (a_{l,h}),\quad
    a_{l,h}(x+N_{l,h})=ax+N_{l,h}$ is a left universal representation
    of $\mathfrak{A}$.
    \item[(vi)]The map $\pi_r:\mathfrak{A}\rightarrow B(\mathfrak{A}_r),\quad a\mapsto (a_{r,h}),\quad
    a_{r,h}(x+N_{r,h})=ax+N_{r,h}$ is a right universal anti-representation
    of $\mathfrak{A}$.
 \end{enumerate}\end{lemma}
 �\begin{proof}�
�(i)   $\pi_{l,h}$� is a continuous homomorphism. Since
�\begin{eqnarray*}�
�\pi_{l,h}(ab)(x+N_{l,h})&=&(ab)_{l,h}(x+N_{l,h})\\�
�&=&(ab)x+N_{l,h}=(a_{l,h})(bx+N_{l,h})\\�
�&=&\pi_{l,h}(a)\circ(\pi_{l,h}(b)(x+N_{l,h})�
�\end{eqnarray*}�
Also
�\begin{eqnarray*}�
�||\pi_{l,h}(a)||&=&\sup\{||\pi_{l,h}(a)(x+N_{l,h})||:||x+N_{l,h}||\leq1\}\\�
�&=&\sup\{||ax+N_{l,h}||:||x+N_{l,h}||\leq1\}\\�
�&\leq&||a+N_{l,h}||\leq||a||�
�\end{eqnarray*}�

�(ii)  $\pi_{r,h}$� is a continuous anti-homomorphism. Since
�\begin{eqnarray*}�
�\pi_{r,h}(ab)(x+N_{r,h})&=&(ab)_{r,h}(x+N_{r,h})\\�
�&=&x(ab)+N_{r,h}=(b_{r,h})(ax+N_{r,h})\\�
�&=&\pi_{r,h}(b)\circ(\pi_{r,h}(a)(x+N_{r,h})�
�\end{eqnarray*}�
Also
�\begin{eqnarray*}�
�||\pi_{r,h}(a)||&=&\sup\{||\pi_{r,h}(a)(x+N_{r,h})||:||x+N_{r,h}||\leq1\}\\�
�&=&\sup\{||xa+N_{r,h}||:||x+N_{r,h}||\leq1\}\\�
�&\leq&||a+N_{r,h}||\leq||a||�
�\end{eqnarray*}�
�(iii)�  �$\|(a_h)\|_l=(\sum_{h\in \A^*}\|a_h\|^2_{l,h})^{\frac{1}{2}}$� is an algebraic  norm. Since
�\begin{eqnarray*}
�\|(a_h)(b_h)\|_l&=&(\sum_{h\in \A^*}\|a_hb_h\|^2_{l,h})^{\frac{1}{2}}\\�
�&\leq&(\sum_{h\in \A^*}\|a_h\|^2_{l,h})^{\frac{1}{2}}(\sum_{h\in \A^*}\|b_h\|^2_{l,h})^{\frac{1}{2}}\\�
�&\leq&\|(a_h)\|_l\|(b_h)\|_l�
�\end{eqnarray*}�

�(iv)�is similar to�(iii).

�(v)   A universal representation is direct sum over all $h\in \A^*$.

�(vi)is similar to�(v).
�\end{proof}�

\begin{lemma}\label{autoperiodic} The following are equivalent for a continuous linear functional
$h$ on a normed algebra $\mathfrak{A}$:
\begin{enumerate}
  \item $h$ is right-autoperiodic;
  \item  $\mathfrak{A}_{r,h}$ is reflexive;
  \item for every bounded sequence $(x_n)$ and left-$h$-bounded sequence $(y_m)$ in $\mathfrak{A}$ the two repeated limits of the double sequence $(\langle x_ny_m, h\rangle )$ are equal provided they both exist.
\end{enumerate}
\end{lemma}
\begin{proof}Refer to \cite[Lemma6]{y1}.

\end{proof}
\begin{lemma}
If $\A$ and $\mathfrak{B}$ are Banach algebras $\phi:\A\rightarrow \mathfrak{B}$ is a continuous epimorphism and $h$ is autoperiodic on $\mathfrak{B}$ then $\phi^*h$ is autoperiodic on $\A$.
\end{lemma}
\begin{proof}
Refer to \cite{y1}.
\end{proof}
\begin{lemma}
Let $\A$ be a $*$-algebra then  $(\A^*)^+\subseteq O(\A)$.
\end{lemma}
\begin{proof}
Let $h\in (\A^*)^+$. Then $\A_{r,h}$ is a Hilbert space. By Lemma\ref{autoperiodic} $h$ is autoperiodic.
\end{proof}
\section{\color{SEC}  GNS-Construction of  General  Banach Algebras  }
\indent
\begin{lemma}
Let $\mathfrak{A}$ be a Banach algebra and  there is a reflexive Banach space $E$, a representation $\pi :\mathfrak{A}\longrightarrow B(E)$. Then there is a extension  $\tilde\pi :WAP(\mathfrak{A})^*\longrightarrow B(E)$ such that $\tilde\pi\circ\iota=\pi$.Where $\iota$ is the canonical map $\iota: \mathfrak{A}\longrightarrow WAP(\mathfrak{A})^*$.
\end{lemma}

\begin{corollary}
\begin{enumerate}
\item [(i)]\cite{GNS} Let $\mathfrak{A}$ be a $C^*$ algebra. Then
there exist a Hilbert space $H$ such that $\mathfrak{A}\subseteq
B(H)$.
    \item [(ii)]\cite{kas} Let $\mathfrak{A}$ be an Arens regular Banach
    algebra.
   Then there exist a reflexive Banach space $E$  such that $\mathfrak{A}\subseteq\mathfrak{A}^{**}\subseteq B(E)$.

        \item[(iii)]\cite{y1} Let $\mathfrak{A}$ be a $WAP$-algebra. Then there
        exist a reflexive Banach space $E$  such that $\mathfrak{A}\subseteq B(E)$.

\end{enumerate}
\end{corollary}
\begin{corollary}\label{se}
Let $\mathfrak{A}$ be an Arens regular Banach algebra. Then
    $\mathfrak{A}^{**}$ is a $WAP$-algebra.
\end{corollary}

\begin{theorem}Let $\mathfrak{A}$ be a Banach algebra . Then there
exist a Banach left [resp. right] $\mathfrak{A}$- module $\mathfrak{A}_l$[resp. $\mathfrak{A}_r$] with a left [resp. right]
norm-decreasing representation \break $\pi_l:\mathfrak{A}\rightarrow
B(\mathfrak{A}_l)$ [resp. anti- representation  $\pi_r:\mathfrak{A}\rightarrow B(\mathfrak{A}_r)$].

\end{theorem}
�\begin{proof}�
For all  �$a\in\A$� and� $(b_{l,h})\in \A_l$� define the module action� $a.(b_{l,h})=(a.b_{l,h})$.� Then
�\[||a.(b_{l,h})||_l=\|(ab_{l,h})\|_l=(\sum_{h\in \A^*}\|ab_{l,h}\|^2_{l,h})^{\frac{1}{2}}\leq||a||(\sum_{h\in \A^*}\|b_h\|^2_{l,h})^{\frac{1}{2}}\]�
�\end{proof}�
\begin{proposition}

Let $\mathfrak{A}$ be a Banach algebra and $B$  be  non trivial
subspace of $\mathfrak{A}^*$. Then:
\begin{enumerate}
\item [(i)] If $\mathfrak{A}_{l,B}:=\oplus_{h\in B} \mathfrak{A}_{l,h}$,
then $\mathfrak{A}_{l,B}$ with  point wise multiplication and the
norm
     $$\|(a_h)\|_{l,B}=(\sum_{h\in B}\|a_h\|^2_{l,h})^{\frac{1}{2}}$$ is a Banach space.
    \item [(ii)]If $\mathfrak{A}_{r,B}:=\oplus_{h\in B} \mathfrak{A}_{r,h}$, then $\mathfrak{A}_{r,B}$ with  point wise multiplication and the  norm $$\|(a_h)\|_{r,B}=(\sum_{h\in B}\|a_h\|^2_{r,h})^{\frac{1}{2}}$$ is a Banach space.
    \item[(iii)]The map $\pi_{l,B}:\mathfrak{A}\rightarrow B(\mathfrak{A}_{l,B}),\quad a\mapsto
    (a_{l,h}),\quad
    a_{l,h}(x+N_{l,h})=ax+N_{l,h}$ is a left universal representation
    of $\mathfrak{A}$.
    \item[(v)] The map $\pi_{r,B}:\mathfrak{A}\rightarrow B(\mathfrak{A}_{r,B}),\quad a\mapsto (a_{r,h}),\quad
    a_{r,h}(x+N_{r,h})=ax+N_{r,h}$ is a right universal anti-representation
    of $\mathfrak{A}$.
\end{enumerate}
\end{proposition}
\begin{proof}It is similar to \ref{zxc}.
\end{proof}
\begin{corollary}
Let $\mathfrak{A}$ be a Banach algebra and $B=WAP(\mathfrak{A})$ be
non-trivial. Then there exist a Banach left [resp. right] $\mathfrak{A}$- module $\mathfrak{A}_{l,B}$[resp. $\mathfrak{A}_{r,B}$ ] with a
left [resp. right] norm-decreasing representation $\pi_l:\mathfrak{A}\rightarrow B(\mathfrak{A}_{l,B})$[resp. anti- representation
$\pi_r:\mathfrak{A}\rightarrow B(\mathfrak{A}_{r,B})$].
\end{corollary}
We now state the main result of this paper.
\begin{corollary}
Let $\mathfrak{A}$ be a Banach algebra and $B\subseteq O(\mathfrak{A})$ be
a subspace of $\mathfrak{A}^*$such that separates the points of $\mathfrak{A}$.  Then there exist a reflexive  Banach space $E$ with a norm-decreasing representation $\pi:\mathfrak{A}\rightarrow B(E)$.
\end{corollary}

\indent

\noindent{{\bf Acknowledgments.}}  This research was supported by
the Center of Excellence for Mathematics at Isfahan university.
\bibliographystyle{amsplain}

\vspace{7mm}

 {\footnotesize
 \noindent
  B.Khodsiani and A. Rejali\\
  Department of Mathematics,
   University of Isfahan,
    Isfahan, Iran\\
    b\_khodsiani@sci.ui.ac.ir\\rejali@sci.ui.ac.ir

\end{document}